\pdfoutput=1
\RequirePackage{ifpdf}
\ifpdf % We~are running pdfTeX in pdf mode
\documentclass[pdftex]{sigma}
\else
\documentclass{sigma}
\fi

\usepackage[utf8]{inputenc}
\usepackage[russian,english]{babel}

\usepackage{array}
\usepackage{amscd}
\usepackage{tikz}
\usetikzlibrary{decorations.markings}
\usepackage{mathtools}

\newcommand{\OP}[1]{\mathrm{#1}}
\newcommand{\CC}{\mathbb{C}}

\newcommand{\RR}{\mathbb{R}}
\newcommand{\PP}{\mathbb{P}}
\newcommand{\TT}{\mathbb{T}}

\newcommand{\YY}{\mathbb{Y}}
\newcommand{\ZZ}{\mathbb{Z}}
\newcommand{\LL}{\mathcal{L}}
\renewcommand{\SS}{\mathbb{S}}

\newcommand{\cp}{\mathbb{CP}}
\newcommand{\rp}{\mathbb{RP}}

\newcommand{\transv}{\mathrel{\text{\tpitchfork}}}

\newcommand{\chek}{\textnormal{\selectlanguage{russian}Ч}}

\makeatletter
\newcommand{\tpitchfork}{%
 \vbox{
  \baselineskip\z@skip
  \lineskip-.52ex
  \lineskiplimit\maxdimen
  \m@th
  \ialign{##\crcr\hidewidth\smash{$-$}\hidewidth\crcr$\pitchfork$\crcr}
 }%
}
\makeatother

\newcommand{\figlet}[3]{
 \begin{figure}[htb]
  \begin{center}
   \begin{tikzpicture}
    #3
   \end{tikzpicture}
   \caption{#1}
   \label{#2}
  \end{center}
 \end{figure}
}

\numberwithin{equation}{section}

\newtheorem{Theorem}{Theorem}[section]
\newtheorem{Corollary}[Theorem]{Corollary}
\newtheorem{Lemma}[Theorem]{Lemma}

\newtheorem{Conjecture}[Theorem]{Conjecture}
 { \theoremstyle{definition}
\newtheorem{Definition}[Theorem]{Definition}
\newtheorem{Definitions}[Theorem]{Definitions}

\newtheorem{Remark}[Theorem]{Remark}
\newtheorem{Question}[Theorem]{Question}
\newtheorem{Questions}[Theorem]{Questions}
}

\begin{document}

\allowdisplaybreaks

\newcommand{\arXivNumber}{2408.14883}

\renewcommand{\PaperNumber}{109}

\FirstPageHeading

\ShortArticleName{Lagrangian Surplusection Phenomena}

\ArticleName{Lagrangian Surplusection Phenomena}

\Author{Georgios DIMITROGLOU RIZELL~$^{\rm a}$ and Jonathan David EVANS~$^{\rm b}$}

\AuthorNameForHeading{G.~Dimitroglou Rizell and J.D.~Evans}

\Address{$^{\rm a)}$~Department of Mathematics, Uppsala Universitet, Uppsala, Sweden}
\EmailD{\href{mailto:georgios.dimitroglou@math.uu.se}{georgios.dimitroglou@math.uu.se}}
\URLaddressD{\url{https://www.uu.se/en/contact-and-organisation/staff?query=N7-1534}}

\Address{$^{\rm b)}$~Department of Mathematics and Statistics, Lancaster University, Bailrigg, UK}
\EmailD{\href{mailto:j.d.evans@lancaster.ac.uk}{j.d.evans@lancaster.ac.uk}}
\URLaddressD{\url{https://jde27.uk}}

\ArticleDates{Received September 03, 2024, in final form November 23, 2024; Published online December 06, 2024}

\Abstract{Suppose you have a family of Lagrangian submanifolds $L_t$ and an auxiliary Lagrangian $K$. Suppose that $K$ intersects some of the $L_t$ more than the minimal number of times. Can you eliminate surplus intersection (surplusection) with all fibres by performing a Hamiltonian isotopy of $K$? Or will any Lagrangian isotopic to $K$ surplusect some of the fibres? We argue that in several important situations, surplusection cannot be eliminated, and that a better understanding of surplusection phenomena (better bounds and a clearer understanding of how the surplusection is distributed in the family) would help to tackle some outstanding problems in different areas, including Oh's conjecture on the volume-minimising property of the Clifford torus and the concurrent normals conjecture in convex geometry. We pose many open questions.}

\Keywords{symplectic geometry; Lagrangian intersections; Floer theory}

\Classification{53D12; 53D40}

\section{Introduction}

In this paper, we draw attention to a class of Lagrangian
intersection problems which we believe deserve further study. We
will start by outlining the general class of problems, and then
give some specific examples where we either know or suspect that
the answer is interesting.

\begin{Definitions} Let $(X,\omega)$ be a symplectic
manifold. Given Lagrangian submanifolds~$K$ and~$L$ in~$X$, define
their {\em geometric intersection number} to be the smallest
number of (transverse) intersections that can be achieved
between $L$ and a Lagrangian Hamiltonian isotopic to $K$:
\[i(K,L)=\min\{\# (\phi(K)\transv L)\mid
  \phi\in\OP{Ham}(X,\omega)\}.\] We say that $K$ and $L$ {\em
  surplusect} if $\#(K\cap L) > i(K,L)$. Given a family
$\LL = \{L_t\mid t\in T\}$ of Lagrangian submanifolds
(parametrised by a measure space $T$ with measure $\mu$) and an
auxiliary Lagrangian $K$, we define the {\em surplusection
  locus} to~be\looseness=-1
\[\SS_\LL(K) = \{t\in T\mid \#(K\cap L_t) > i(K,L_t)\} \subseteq
 T\] and the {\em mean surplusection} to be the {\samepage
integral \[\frac{1}{\int_T {\rm d}\mu}\int_T (\#(K\cap L_t) - i(K,L_t)) {\rm d}\mu,\]
assuming it is well defined.}
\end{Definitions}

In Figure \ref{fig:surplusection}, we see an example of a
Lagrangian submanifold $K$ which can be straightened by a
Hamiltonian isotopy to remove all of its surplusection with the
vertical Lagrangians. Our main thesis is that, in many
situations of interest, you cannot achieve this, and
surplusection is forced upon you no matter how you isotope
$K$.

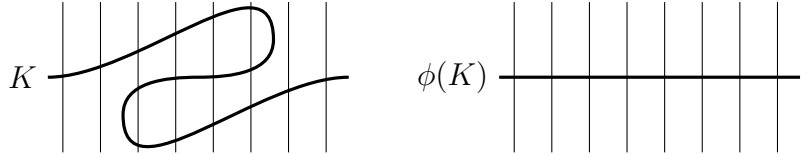
\figlet{(1) The Lagrangian $K$ surplusects some of the vertical
 fibres, (2) but this surplusection can be removed by an isotopy of
$K$.}{fig:surplusection}{\draw[very thick] (-1,0) to[out=0,in=90] (2,0.5)
to[out=-90,in=0] (1,0) to[out=180,in=90] (0,-0.5)
to[out=-90,in=180] (3,0);
\foreach \x in {-0.8,-0.3,0.2,0.7,1.2,1.7,2.2,2.7}{
 \draw (\x,-1) -- (\x,1);
}
\node at (-1,0) [left] {$K$};
\begin{scope}[shift={(6,0)}]
 \draw[very thick] (-1,0) -- (3,0);
\foreach \x in {-0.8,-0.3,0.2,0.7,1.2,1.7,2.2,2.7}{
 \draw (\x,-1) -- (\x,1);
}
\node at (-1,0) [left] {$\phi(K)$};
\end{scope}
}

\begin{Questions}\quad
\begin{enumerate}\itemsep=0pt
\item[(A)] Is there a $\phi\in\OP{Ham}(X,\omega)$ for which
 $\SS_\LL(\phi(K))$ is empty? In other words, can we
 arrange for $\phi(K)$ to simultaneously intersect all
 $L_t$ in the family minimally?
\item[(B)] If not, what is
 $\inf\{\mu(\SS_\LL(\phi(K)))\mid \phi\in\OP{Ham}(X,\omega)\}$? Or can we bound from below
 the mean surplusection?
\end{enumerate}
\end{Questions}

In some cases, the quantity $i(K,L)$ might be hard to compute,
in which case one could replace it by an Ersatz which bounds
$i(K,L)$ from below (like the rank of Lagrangian Floer
cohomology, assuming it is defined between $K$ and $L$). In the
rest of this note, we outline some of the situations in which
these very general questions arise naturally.

\section{Volume bounds}

\subsection{Crofton formula}
L\^{e} \cite{HVLe} proved the following remarkable
formula\footnote{It is named the Crofton formula after an
 analogous result in integral geometry of Euclidean space. Oh
 \cite[p.~503]{Oh90} mentions the existence of such a formula
 for Lagrangian submanifolds of $\cp^n$, attributing it to
 Kleiner. The proof is all the more remarkable for its
 simplicity: we sketch a proof following \cite[Section~1.1]{Rota}. Observe that both sides of the equation are
 additive under taking disjoint union, e.g.,
 $\OP{vol}(K_1\coprod K_2)=\OP{vol}(K_1)+\OP{vol}(K_2)$. Since one can approximate
 any Lagrangian arbitrarily closely by something which is
 piecewise linear, made of tiny patches of totally geodesic
 $\rp^n$s, it is sufficient to prove the formula for totally
 geodesic $\rp^n$s. Since these are all related by
 isometries, this just amounts to fixing the constant
 $\xi_n$.} for Lagrangian submanifolds of $\cp^n$:
\begin{equation}\label{eq:crofton}
\OP{vol}(K) = \xi_n \int_{g\in
  {\rm PU}(n+1)} \#(K\cap gL){\rm d}\mu,
  \end{equation}
where $\OP{vol}(K)$ is the unsigned Riemannian volume of
$K$, $\xi_n$ is a constant depending only on~$n$,~${\rm PU}(n+1)$ is the group of projective unitary isometries of
$\cp^n$, $L$ is the standard ${\rp^n\subset\cp^n}$, and
$\mu$ is the Haar measure on ${\rm PU}(n+1)$. We can compute the
number $\xi_n$ by picking ${K=\rp^n}$: since\footnote{Oh \cite{OhTight}
 calls this property {\em global tightness}, and showed that
 the only globally tight Lagrangian submanifolds of $\CC\PP^n$ are the
 standard $\RR\PP^n$s.}
$\#(\rp^n\cap g\rp^n) = n+1$ for all but a measure-zero set of~${g\in {\rm PU}(n+1)}$, we get $\xi_n = \OP{vol}(\rp^n)/(n+1)$.
We will normalise the Fubini--Study metric by assuming that $\OP{vol}(\rp^n)$ is half the surface area of a unit Euclidean
$n$-sphere, that is,
\begin{align}\label{eq:xi_n}
 \OP{vol}(\rp^n)=\frac{\pi^{\frac{n+1}{2}}}{\Gamma((n+1)/2)},
 \qquad \xi_n = \frac{\pi^{\frac{n+1}{2}}}{(n+1)\Gamma((n+1)/2)},
\end{align}
where $\Gamma$ is the Gamma function. So, for an arbitrary
Lagrangian $K$ for which $i(K,\RR\PP^n)$ is known, the
problem of bounding from below the mean surplusection of
$\phi(K)$ and $\{g\rp^n\mid g\in {\rm PU}(n+1)\}$ is equivalent to
finding a lower bound on $\OP{vol}(\phi(K))$ (assuming we know
$i(K,L)$). Bounding the volume of $\phi(K)$ from below is a
notoriously thorny problem in general: Oh {\cite[p.\ 192]{Oh93}}
conjectured that the monotone Clifford torus minimises volume
amongst Lagrangians in its Hamiltonian isotopy class, but this
conjecture remains open thirty years later.

\begin{Definition} Let $\mu\colon\cp^n\to \RR^n$ be the
moment map for the standard torus action, normalised so that its
image is the simplex with its vertices at the origin and the
standard basis vectors. The monotone Clifford torus is the
Lagrangian torus
\smash{$\TT^n=\mu^{-1}\bigl(\tfrac{1}{n+1},\ldots,\tfrac{1}{n+1}\bigr)$}. It
is the unique nondisplaceable fibre of the moment map. We say
that a Lagrangian torus is {\em Clifford-type} if it has the
form $K=\phi(\TT^n)$ for some Hamiltonian symplectomorphism
$\phi$.
\end{Definition}

\subsection{The bound of Alston and Amorim}\label{AA}

Alston
\cite{Alston} and Alston--Amorim \cite{AlstonAmorim} used the
Crofton formula \eqref{eq:crofton} to give lower bounds on the
volume of a Clifford-type torus $\phi(\TT^n)$ by showing
that\footnote{If $n$ is odd, then this bound is established by
 calculating (characteristic~2) Floer cohomology between
 $\rp^n$ and (a suitable local system on) $K$; if $n$ is
 even, then all such Floer cohomology groups vanish, and the
 bound on geometric intersection is proved instead using the
 Abreu--Macarini trick \cite{AbreuMacarini}, that is, by
 computing Floer cohomology between $\rp^n\times\rp^n$ and
 $K\times K$ in $\cp^n\times\cp^n$.}
\begin{equation}\label{eq:aa_geom_int}i(\TT^n,\rp^n) \geq 2^{\lceil
  n/2 \rceil}.\end{equation}
The resulting volume bound is
\begin{equation}\label{eq:aa}
 \OP{vol}(\TT^n)\geq \frac{\pi^{\frac{n+1}{2}}2^{\lceil
   \frac{n}{2}\rceil}}{(n+1)\Gamma((n+1)/2)}.
\end{equation}
If $n=1$, then the lower bound from \eqref{eq:aa} is sharp
because it equals $\OP{vol}\bigl(\TT^1\bigr)=\pi$. But if $n\geq 2$,
the bound is substantially lower than
\begin{equation}\label{eq:volt}
 \OP{vol}(\TT^n)=\frac{(2\pi)^n}{(n+1)^{\frac{n+1}{2}}},
\end{equation}
as you can see from Table \ref{tab:bounds} below. Since
$\TT^n$ is conjectured to minimise volume in its Hamiltonian
isotopy class, this suggests that there should be lots of
surplusection.

\bgroup
\def\arraystretch{1.2}
\begin{table}
 \begin{center}
  \begin{tabular}{p{0.5cm}>{\raggedleft\arraybackslash}p{2.9cm}>{\raggedleft\arraybackslash}p{2.5cm}>{\raggedleft\arraybackslash}
   p{2.5cm}}
   $n$ & Alston--Amorim & Goldstein & $\OP{vol}(\TT^n)$\\
   \hline
   $1$ & $3.14159\dots$ & $3.14159\dots$ & $3.14159\dots$\\
   $2$ & $4.18879\dots$ & $7.25519\dots$ & $7.59762\dots$\\
   $3$ & $9.86960\dots$ & $13.95772\dots$ & $15.50313\dots$\\
   $4$ & $10.52757\dots$ & $23.54038\dots$ & $27.88010\dots$\\
   $5$ & $20.67085\dots$ & $35.80296\dots$ & $45.33624\dots$\\
   $6$ & $18.89906\dots$ & $64.93939\dots$ & $67.80099\dots$
  \end{tabular}
  \caption{Lower bounds on $\OP{vol}(\phi(\TT^n))$ for
   $\phi\in\OP{Ham}(\cp^n)$, due to Alston--Amorim and
   Goldstein, com\-pared with the volume of the standard
   Clifford torus (final column).}
  \label{tab:bounds}
 \end{center}
\end{table}
\egroup

\begin{Remark} In fact, as we will see in a moment, there
exist much stronger bounds (due to Goldstein) on the volume of a
Clifford-type torus (also shown in Table~\ref{tab:bounds}). However, the bound~\eqref{eq:aa_geom_int} on
$i(K,\rp^n)$ is actually optimal: the lower bound of
$2^{\lceil n/2\rceil}$ intersections is realised by
intersecting $\TT^n$ with a specific $\rp^n$. Namely, let
$L$ be the fixed point set of the antisymplectic involution
\[[z_0:z_1:\dots :z_{n-1}:z_n] \mapsto
 \begin{cases}
  [\bar{z}_1:\bar{z}_0:\dots :\bar{z}_n:\bar{z}_{n-1}]&\mbox{if }n\mbox{ is odd},\\
  [\bar{z}_1:\bar{z}_0:\dots :\bar{z}_{n-1}:\bar{z}_{n-2}:\bar{z}_n]&\mbox{if }n\mbox{ is even}.
 \end{cases}\] This intersects $\TT^n$ in a subtorus of
dimension $\lceil n/2\rceil$, so if $g\in {\rm PU}(n+1)$ is a
generic element sufficiently close to the identity then
$\TT^n\cap gL$ consists of $2^{\lceil n/2\rceil}$ transverse
intersections.
\end{Remark}

\subsection{Goldstein's bound}\label{sct:goldstein}
In fact, Goldstein \cite{Goldstein03} gave a considerably better
bound by exploiting the following identity from integral
geometry:
\[%\label{eq:howard}
 \OP{vol}(K)\cdot\OP{vol}(L) = \zeta_n\int_{{\rm PU}(n+1)}\#(K\cap
 gL) {\rm d}g,
\]
where $K$ and $L$ are Lagrangian submanifolds of $\cp^n$
and $\zeta_n$ is a constant (depending only on~$n$) to be
determined shortly. This identity is attributed to Howard
\cite{Howard}.\footnote{We can see how to prove it as follows:
 fix $K$ and treat both sides as functionals of $L$,
 additive over concatenation. Just as with Crofton's formula,
 the identity then holds with a constant $\zeta_K$
 potentially depending on $K$. But since both
 $\OP{vol}(K)\cdot\OP{vol}(L)$ and the integral are symmetric in
 $K$ and $L$, we get the same constant $\zeta_K=\zeta_L$,
 and so the constant depends only on the choice of ambient
 space.} To determine the constant, Goldstein applies this
inequality in the case $K=L=\rp^n$
\[
\OP{vol}(\rp^n)^2 = \zeta_n(n+1),\qquad\mbox{so}\qquad
 \zeta_n=\frac{\pi^{n+1}}{(n+1)\left(\Gamma((n+1)/2)\right)^2}.
 \]
Now he applies it when $K=L=\phi(\TT^n)$ is a Clifford-type
torus. Since $HF(\phi(\TT^n),\phi(\TT^n))$ has rank $2^n$,
this gives the bound
\[\OP{vol(\phi(\TT^n))}\geq\frac{1}{\Gamma((n+1)/2)}
 \sqrt{\frac{2^n\pi^{n+1}}{n+1}}\] (tabulated for small $n$
in Table \ref{tab:bounds}).

\begin{Remark}
  The fact that the Goldstein bound is stronger than the
Alston--Amorim bound means that {\em surplusection must occur}
between $\TT^n$ and $\{g\rp^n\mid g\in {\rm PSU}(n+1)\}$. But it
seems like an~interesting problem to determine how this
surplusection is distributed over ${\rm PSU}(n+1)$: the volume bound
only guarantees that the {\em mean} surplusection is large, but
that could mean a very small surplusection locus with very high
surplusection. We take a moment, therefore, to study the
distribution of $\#\bigl(\TT^2\cap g\rp^2\bigr)$ for $g\in {\rm PSU}(3)$ where
$\TT^2$ is the Clifford torus in its standard (conjecturally
minimal) position. This is the smallest interesting case, and
the only one we have solved completely.
\end{Remark}

\subsection{Crofton distribution for the Clifford torus}

A point $(x,y,z)\in\CC^3$ represents a point $[x:y:z]$
on the Clifford torus if and only if
$|x|^2=|y|^2=|z|^2$. Given a matrix $g\in {\rm SU}(3)$, a point
$(x,y,z)\in\CC^3$ represents a point $[x:y:z]\in g\rp^2$ if
and only if there exists a unit complex number $u$ such that
$(ux,uy,uz)$ is contained in the real span of the columns of
$g$. Therefore, the intersection
$\TT^2\cap g\rp^2$ consists of points $[x:y:z]$ with~${(x,y,z)}$ in the real span of the columns of $g$ and which
satisfy $|x|^2=|y|^2=|z|^2$. Write~$g_{ij}$ for the entries
of $g$ and let $a=(a_1,a_2,a_3)\in\RR^3$. Let
$(x_1,x_2,x_3)=ga^{\mathsf{T}}$ be the real linear combination of the
columns of $g$ with coefficients given by $a_1$, $a_2$, $a_3$. We
have \[|x_i|^2=\sum_{j,k}g_{ij}\bar{g}_{ik}a_ja_k.\]
The intersection $\TT^2\cap g\rp^2$ is therefore in
bijection with the subvariety of $\rp^2$ given by the
intersection of the two conics
\begin{align*}
 \sum_{j,k}(g_{2j}\bar{g}_{2k}-g_{1j}\bar{g}_{1k})a_ja_k=0,\qquad
 \sum_{j,k}(g_{3j}\bar{g}_{3k}-g_{1j}\bar{g}_{1k})a_ja_k=0.
\end{align*}
In particular, this gives an upper bound of four intersection
points (by B\'{e}zout's theorem). The lower bound of two
intersections comes from Floer cohomology, and there must always
be an even number of intersections if the intersection is
transverse, so the only possibilities for
$\#\bigl(\TT^2\cap g\rp^2\bigr)$ if the intersection is transverse
are $2$ and $4$.

\begin{Corollary} The probability that a randomly
 chosen $\rp^2$ intersects the standard Clifford torus at
 $2$ $($respectively $4)$ points is
 \smash{$2-\frac{\pi}{\sqrt{3}}\approx 0.186$} $\bigl($respectively
 \smash{$\frac{\pi}{\sqrt{3}}-1\approx 0.814\bigr)$}.
\end{Corollary}

\begin{proof}
 The volume of the standard Clifford torus is
 \smash{$\frac{4\pi^2}{3\sqrt{3}}$} and the volume of $\rp^2$ is
 $2\pi$. The expected number of intersections is given by the
 Crofton formula
 \[
 \frac{4\pi^2}{3\sqrt{3}} = \frac{2\pi}{3}(2p_2+4p_4),
 \]
 where $p_i$ denotes the probability of $i$
 intersections. We also know that $p_2+p_4=1$ because the
 probability of a non-transverse intersection is zero by Sard's
 theorem and the only possible numbers for transverse
 intersections are $2$ or $4$. This gives a pair of
 simultaneous equations for~$p_2$ and $p_4$ which have the
 solutions as stated.
\end{proof}

\subsection{Questions}

The analogous problem in higher dimensions
can also be recast as an intersection of real quadrics. Just
knowing the volume and the fact that the probabilities sum to
$1$ is no longer enough to determine all the probabilities
(for example, for $\cp^3$ we need to determine~$p_4$,~$p_6$
and~$p_8$).

\begin{Question} Fix $n>3$. Let $p_{2i}$ be the
probability that the standard Clifford torus
$\TT^n\subset\CC\PP^n$ intersects $g\RR\PP^n$ transversely
in $2i$ points (where $g\in {\rm PSU}(n+1)$ and we use the Haar
measure on ${\rm PSU}(n+1)$). Determine the probability distribution
$p_{2^{\lceil n/2\rceil}},\ldots,p_{2^n}$ for each $n$.
\end{Question}

\begin{Remark}
  Equations \eqref{eq:crofton}, \eqref{eq:xi_n} and
\eqref{eq:volt} give the mean of this probability distribution
as
\begin{align*}
 \sum_{m=2^{\lceil n/2\rceil}}^{2^n} mp_m
 =\xi_n^{-1}\OP{vol}(\TT^n) =
               \frac{(2\pi)^n}{(n+1)^{\frac{n+1}{2}}} \cdot
               \frac{(n+1)\Gamma((n+1)/2)}{\pi^{\frac{n+1}{2}}}
              =\frac{2^n\pi^{\frac{n-1}{2}}\Gamma((n+1)/2)}{(n+1)^{\frac{n-1}{2}}}.
\end{align*}
As observed by Alston \cite{Alston}, using Stirling's
approximation
\smash{$\Gamma(x)\approx
\sqrt\frac{2\pi}{x}\bigl(\frac{x}{{\rm e}}\bigr)^x$}, the mean is
approximately $2^{cn}$, where
$c=\frac{1}{2}\log_2(2\pi/{\rm e})\approx 0.604$.
\end{Remark}

\begin{Question} Can anything more be said about the
asymptotics of this distribution as $n\to \infty$?
\end{Question}

\begin{Question}[refined Oh conjecture] Let
$\phi(\TT^n)$ be a Clifford-type torus for
$\phi\in\OP{Ham}(\cp^n)$, and let $\mu$ be the Haar measure
on ${\rm PSU}(n+1)$. Let $q_{2k}$ be the probability with respect
to $\mu$ that~$\#(\phi(\TT^n)\cap g\rp^n)\geq 2k$. Is it
true that
$\sum_{k\geq \ell}p_{2k}\leq \sum_{k\geq \ell}q_{2k}$ for all $\ell$? Even short of proving this (which
would imply Oh's conjecture), any lower bound (independent of
$\phi$) on some specific sum $\sum_{k\geq \ell}q_{2k}$ would
give more information about how the surplusection guaranteed by
Goldman's bound is distributed.
\end{Question}

\begin{Remark} We note that Viterbo \cite[Proposition~3.13]{ViterboMetric} established a lower bound for the volume
of {\em any} Lagrangian submanifold $L\subset\cp^n$. The lower
bound is a constant times $\tilde{d}(L)^{n/2}$ where
$\tilde{d}(L)$ is the displacement energy of the preimage of
the Lagrangian in $S^{2n-1}\subset\CC^{n+1}$ under the
symplectic reduction $S^{2n-1}\to\cp^n$. This indicates that
surplusection is a ubiquitous phenomenon and should probably be
closely related to action filtrations, spectral invariants, and
persistence modules.
\end{Remark}

\section{Chekanov-type tori}

Having focused on the case of Clifford-type tori in all
dimensions, and how they intersect all images of $\rp^n$ under
K\"{a}hler isometries, we turn to a more specific problem: how
{\em Chekanov-type} tori in $\cp^2$ intersect a particular
{\em loop} of $\rp^2$s. In this case, a Chekanov-type torus
can be displaced from $\rp^2$, so {\em any} intersection is
surplus. We will be able to give better, in fact sharp, control
on the measure of the surplusection locus.

\begin{Definition}[clean loops] For each $t\in[0,1]$, let $L_t$ be
the Lagrangian submanifold
\[%\label{eq:clean_loop}
 L_t \coloneqq \bigl\{ \big[(a_1+{\rm i}a_2){\rm e}^{{\rm i}\pi t/3} : (a_1-{\rm i}a_2){\rm e}^{{\rm i}\pi t/3}:
          a_3{\rm e}^{-{\rm i}2\pi t/3}\big] \mid [a_1:a_2:a_3]
        \in\rp^2\bigr\}\subset\cp^2.
\]
As $t$ varies this defines for us a loop of $\rp^2$s coming
from a 1-parameter subgroup of ${\rm PU}(3)$. If
$t_1\neq t_2\mod 1$, then the intersection
$L_{t_1}\cap L_{t_2}$ consists of a single transverse
intersection at~${[0:0:1]}$ and a circle of clean intersection
$\{[a:b:0]\mid |a|=|b|\}$. We call this a {\em clean loop} of~$\rp^2$s.
\end{Definition}

Consider the moment map
\[
\mu\colon\ \cp^2\to\RR^2,\qquad \mu([z_1:z_2:z_3]) =
 \frac{1}{2\sum_{j=1}^3|z_j|^2}\bigl(|z_1|^2,|z_2|^2\bigr)
 \]
for the standard torus action. Let $\Delta\subset\RR^2$ be the
moment image and $\ell$ be the intersection of the diagonal
$\{(x,x)\mid x\in\RR\}$ with $\Delta$. Similarly to Chekanov
and Schlenk \cite{CS}, consider the singular hypersurface
$\mu^{-1}(\ell)$, which is preserved by the circle action
$\big[ {\rm e}^{{\rm i}\theta}z_1 : {\rm e}^{-{\rm i}\theta}z_2 :z_3 \big]$. The
symplectic reduction of $\mu^{-1}(\ell)$ by this circle action
is a singular sphere which we will call $S$; it has two
singular points $p$ and $q$ living over the ends of
$\ell$. Here, $p$ is over the endpoint where $\ell$ meets
the slanted edge of $\Delta$ and $q$ is over the endpoint
where $\ell$ hits the lower-left corner of $\Delta$. The
preimage of $p$ under reduction is a circle on which the
circle action has stabiliser $\ZZ_2$; the preimage of $q$ is
a fixed point of the circle action. The submanifold
$L_t\subset\mu^{-1}(\ell)$ projects to an arc
$\alpha_t\subset S$ connecting $p$ and $q$. See Figure
\ref{fig:rp2s}.

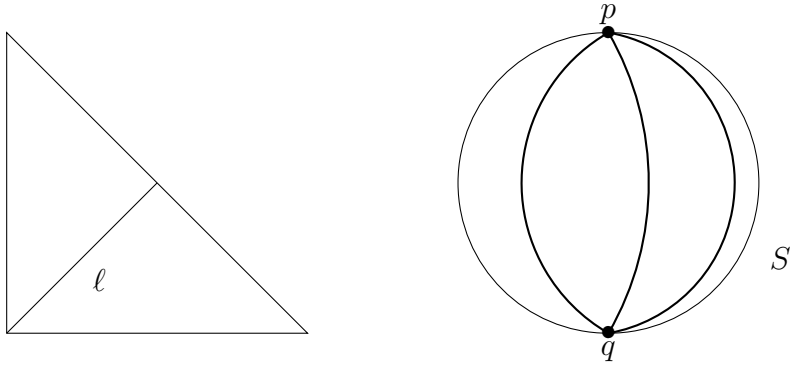
\figlet{Left: The moment image of $\cp^{2}$ and the diagonal
 arc $\ell$. Right: The symplectic reduction $S$ of
 $\mu^{-1}(\ell)$ and several of the arcs $\alpha_t$.}{fig:rp2s}{
 \draw (0,0) -- (4,0) -- (0,4) -- cycle;
 \draw (0,0) -- (2,2) node [midway,below right] {$\ell$};
 \begin{scope}[shift={(8,2)}]
  \draw (0,0) circle [radius = 2];
  \node (p) at (0,2) {$\bullet$};
  \node at (p.north) {$p$};
  \node (q) at (0,-2) {$\bullet$};
  \node at (q.south) {$q$};
  \foreach \x in {-60,30,80} {
   \draw[thick] (p.center) arc [radius = {2/sin(\x)},start angle={\x},end angle={-\x}];
  }
  \node at (2,-1) [right] {$S$};
 \end{scope}
}

\begin{Definitions}
  For each loop $\gamma\subset S\setminus \{p,q\}$, there is
a Lagrangian torus
$T_\gamma\subset \mu^{-1}(\ell) \subset \cp^{2}$ which
projects to $\gamma$ along the symplectic reduction. Let $P$
denote the component of $S\setminus\gamma$ containing
$p$. If $\OP{area}(P)=\frac{1}{3}\OP{area}(S)$, then
$T_\gamma$ is monotone. There are two isotopy classes of loop
$\gamma\subset S\setminus \{p,q\}$: loops of class A separate
$p$ from $q$; loops of class B do not. If $\gamma$ has
class~A, then $T_\gamma$ is a Clifford-type torus; if
$\gamma$ has class~B, then $T_\gamma$ is not
Clifford-type. Pick your favourite $\gamma$ of class $B$ and
call $T_\gamma$ the standard Chekanov torus $\chek$. A
{\em Chekanov-type torus} is any Lagrangian torus of the form
$\phi(\chek)$ for $\phi\in\OP{Ham}\big(\cp^2\big)$, and if
$\gamma$ is of class~B, then $T_\gamma$ is Chekanov-type. See
Figure~\ref{fig:clif_chek}.
\end{Definitions}

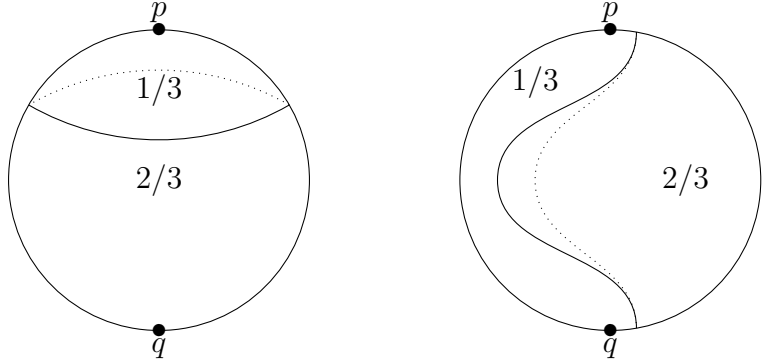
\figlet{A loop of class A (left) and class B (right), yielding
 respectively a monotone Clifford-type and a Chekanov-type torus in
 $\cp^{2}$.}{fig:clif_chek}{
 \draw (0,0) circle [radius = 2];
 \node (p) at (0,2) {$\bullet$};
 \node at (p.north) {$p$};
 \node (q) at (0,-2) {$\bullet$};
 \node at (q.south) {$q$};
 \draw (30:2) arc [radius = 2*cos(30)/sin(30),start angle=-60,end angle=-120];
 \draw[dotted] (30:2) arc [radius = 2*cos(30)/sin(30),start angle=60,end angle=120];
 \node at (0,0) {$2/3$};
 \node at (0,1.2) {$1/3$};
 \begin{scope}[shift={(6,0)}]
  \draw (0,0) circle [radius = 2];
  \node (p) at (0,2) {$\bullet$};
  \node at (p.north) {$p$};
  \node (q) at (0,-2) {$\bullet$};
  \node at (q.south) {$q$};
  \draw (80:2) to[out=-90,in=90] (-1.5,0) to[out=-90,in=90] (-80:2);
  \draw[dotted] (80:2) to[out=-90,in=90] (-1,0) to[out=-90,in=90] (-80:2);
  \node at (-1,1.3) {$1/3$};
  \node at (1,0) {$2/3$};
 \end{scope}}

\begin{Lemma}\label{lma:nearby_tori} For any
 neighbourhood $U$ of $L_t\cup L_{t+1/3}$, there is a
 monotone Chekanov-type torus~$\chek_U$ contained in $U$.
\end{Lemma}

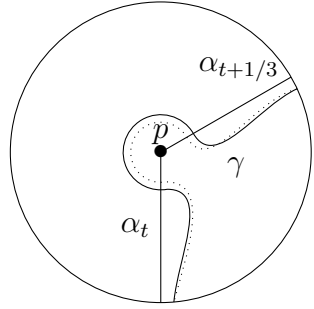
 \figlet{The arcs $\alpha_t$ and $\alpha_{t+1/3}$ and the
  loop $\gamma$, as viewed from the North pole (the dotted line
  indicates how it looks in the Southern hemisphere).}{fig:nearby_tori}{
  \draw (0,0) circle [radius = 2];
  \node (p) at (0,0) {$\bullet$};
  \node at (p.north) {$p$};
  \draw (-90:2) -- (0,0) node [midway,left] {$\alpha_t$};
  \draw (0,0) -- (30:2) node [pos=0.6,above=0.2cm] {$\alpha_{t+1/3}$};
  \draw (25:2) to[out=-155,in=-60] (30:0.5) arc [radius=0.5,start angle=30,end angle = 280] to[out=0,in=85] (-85:2);
  \draw[dotted] (25:2) to[out=-165,in=-60] (30:0.4) arc [radius=0.4,start angle=30,end angle = 280] to[out=0,in=75] (-85:2);
  \node at (1,-0.2) {$\gamma$};
 }

\begin{proof}
 Figure \ref{fig:nearby_tori} shows a type B loop $\gamma$
 chosen in such a way that
 $\OP{area}(P) = \frac{1}{3}\OP{area}(S)$. This can be chosen
 to lie arbitrarily close to $\alpha_t\cup\alpha_{t+1/3}$. If
 it is close enough, then $T_\gamma\subset U$ and we can take
 $\chek_U=T_\gamma$.
\end{proof}

The fact that the Chekanov torus arises as a clean surgery of
two $\rp^2$s in this way was observed by Abreu and Gadbled
\cite{AbreuGadbled}.

\begin{Corollary}\label{cor:intersection}If
 $K=\phi(\chek)$ is a Chekanov-type torus, then the
 intersection $K\cap (L_t\cup L_{t+1/3})$ is
 nonempty for any $t\in[0,1]$. Moreover, if $K$ intersects
 both $L_t$ and $L_{t+1/3}$ transversely and is disjoint
 from $L_t\cap L_{t+1/3}$, then the number of intersection
 points between $K$ and $L_t\cup L_{t+1/3}$ is at least
 four.
\end{Corollary}

\begin{proof}
 If $K$ is disjoint from $L_t\cup L_{t+1/3}$, then it is
 disjoint from some neighbourhood $U$ of this union, and
 hence from the torus $\chek_U$ constructed in Lemma~\ref{lma:nearby_tori}. But since both $K$ and $\chek_U$
 are monotone Chekanov-type tori, the Floer cohomology group
 $HF(K,\chek_U)$ (with a suitable choice of local coefficient
 system) has rank four, so $i(K,\chek_U)\geq 4$, so this is
 impossible.

 If the transversality conditions from the corollary hold, then
 there is a bijection between intersection points
 $K\cap (L_t\cup L_{t+1/3})$ and the intersection
 points $K\cap \chek_U$ for some sufficiently small
 neighbourhood $U$. The lower bound on the number of
 intersection points then follows from the fact that
 $i(K,\chek_U)\geq 4$.
\end{proof}

\begin{Corollary} \label{cor:surplusection} Take
 $\LL = \{L_t\mid t\in[0,1]\}$ and let $K$ be a monotone
 Chekanov-type torus. The surplusection locus
 $\SS_\LL(K)\subset[0,1]$ has Lebesgue measure at least
 $2/3$.
\end{Corollary}

\begin{proof}
 Let $F=[0,1]\setminus \SS_\LL(K)$. If $t\in F$, then
 $t+1/3\mod 1$ and $t-1/3\mod 1$ belong to~$\SS_\LL(K)$
 by Corollary~\ref{cor:intersection}. This means that $F$,
 $F+1/3\mod 1$ and $F-1/3\mod 1$ are three disjoint subsets
 of $[0,1]$, each having the same measure, and hence
 the measure of each is at most $1/3$. Therefore, the measure
 of $\SS_\LL(K)$ is at least $2/3$.
\end{proof}

\begin{Remark}
  On a philosophical level, one may compare Corollary
  \ref{cor:intersection} with the recent works by
  Mak--Smith~\cite{MakSmith} and Polterovich--Shelukhin
  \cite{PolShel}, which exhibit Lagrangian submanifolds which do
  not exhibit rigidity when taken alone but do exhibit rigidity
  when several copies are taken together.
\end{Remark}

\begin{Remark}
  One can use Corollary \ref{cor:surplusection}, together with
standard formulas for integrating over Lie groups, to obtain
volume bounds for the Chekanov torus from Crofton's
formula,\footnote{The interested reader can look at version~1 of
 this paper on arXiv for the details of how this bound is
 computed.} however the resulting bound is only
$\approx 3.27$, which is significantly worse than Goldstein's
bound. Goldstein's bound still applies to the Chekanov torus,
and indeed any wide Lagrangian torus in $\cp^2$ (such as the
Vianna tori), because the rank of Floer cohomology is still
$4$.
\end{Remark}

\begin{Remark}
  Despite the fact that this argument gives a weaker volume
bound than Goldstein, it does give better control on the measure
of the surplusection locus. In fact, it is sharp: one can
(visibly, in the symplectic reduction) disjoin the Chekanov
torus from $\bigcup_{t\in(\epsilon,1/3-\epsilon)}L_t$ for any
$\epsilon$. One can obtain still more information about the
surplusection locus: Gathercole \cite[Theorem~1.2]{Gathercole}
shows that $\phi(\YY)$ cannot be disjoined from
$L_s\cup L_{s+t}$ for any $s\in[0,1]$ and any
$t\in[1/3,2/3]$. This means that if $\phi(\YY)$ misses (say)
$L_0$ then the surplusection locus really contains the whole
interval $[1/3,2/3]$ (not just the endpoints).
\end{Remark}

The argument we used is very {\em ad hoc}, and does not
easily generalise to tackle slight modifications of the
problem. For example, the following.

\begin{Question} Can one see similar phenomena in higher
dimensions?
\end{Question}

Similar arguments appear to break down in $\cp^4$
and above. If we take a loop of diagonal matrices in ${\rm U}(n+1)$,
where the first $k$ entries coincide and the last $m=n+1-k$
entries coincide, and apply this loop to $\rp^n$, then we get a
clean loop of Lagrangian $\rp^n$s which intersect pairwise
along a copy of $\rp^{k-1}$ and a copy of $\rp^{m-1}$. It
appears (though we have not checked) that the clean surgery is
the Oakley--Usher Lagrangian \cite{OakleyUsher}, which has
vanishing Floer cohomology (it is actually displaceable) and our
argument relies crucially on being able to use Floer cohomology
with the result of surgery to get lower bounds on intersections.

\begin{Question} Let $\{g_t\mid t\in[0,1]\}$ be a {\em
 generic} 1-parameter subgroup in ${\rm SU}(3)$ isomorphic to~${\rm U}(1)$, so that $\rp^2\cap g_t\rp^2$ consists of three
points if $t\neq 0,1$. Let
$\LL'=\bigl\{g_t\rp^2\mid t\in [0,1]\bigr\}$ and let $K=\phi(\chek)$
be a monotone Chekanov-type torus. Is it true that
$\SS_{\LL'}(K)$ has Lebesgue measure at least $2/3$? The
surgery of $\rp^2$ with $g_t\rp^2$ is now the non-orientable
(Maslov 1) Lagrangian constructed by Abreu and Gadbled
\cite{AbreuGadbled}.
\end{Question}

\begin{Question} Are we justified in harbouring an
expectation that Chekanov tori should have bigger volume
than Clifford tori? Or can one find Chekanov tori whose volumes
get arbitrarily close to the volume of any Clifford torus?
C. Evans\footnote{No relation.} \cite{CEvans} showed that, under
(volume-decreasing) Lagrangian mean curvature flow with
surgeries, any Chekanov torus contained in the hypersurface
$\mu^{-1}(\ell)$ satisfying a certain symmetry condition
eventually flows with surgeries to the Clifford torus. However,
this is a very special situation with lots of symmetry.
\end{Question}

\section{Concurrent normals}

\subsection{The original conjecture}

We now move to another, completely different situation in
which surplusection appears (and is conjectured to appear even
more). Consider a not-too-eccentric ellipse in the plane and
draw on all the lines which intersect it normally. Most points
lie on two such normals, but there is a~small star-shaped
(astroid) region in the middle of the ellipse through every
point of which four normals pass.

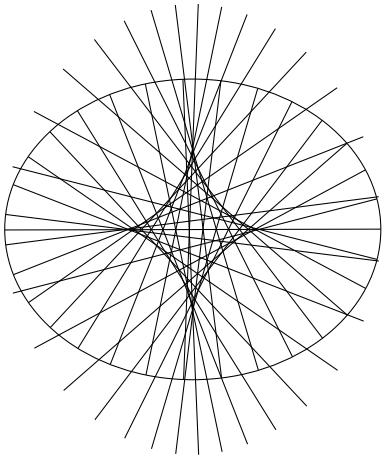
\begin{center}
 \begin{tikzpicture}
  \foreach \a in {2.5} {
   \foreach \b in {2} {
    \foreach \c in {5} {
     \draw (0,0) circle [x radius = \a,y radius = \b];
     \foreach \m in {31}{
      \foreach \n in {-7,-6,...,7} {
       \draw ({\a*cos(360*\n/\m+0.1)},{\b*sin(360*\n/\m+0.1)}) --++
       ({180+atan(\a*tan(360*\n/\m+0.1)/\b)}:\c);
      };
      \foreach \n in {8,9,...,23} {
       \draw ({\a*cos(360*\n/\m+0.1)},{\b*sin(360*\n/\m+0.1)}) --++
       ({atan(\a*tan(360*\n/\m+0.1)/\b)}:\c);
      };
     };
    };
   };
  };
 \end{tikzpicture}
\end{center}

As the ellipse becomes more eccentric, the astroid starts to
protrude out of the top and bottom; as the eccentricity goes to
zero, the astroid shrinks to a point. In higher dimensions, the
normals to ellipsoids \smash{$\sum \frac{x_i^2}{a_i^2}=1$} have a
similar behaviour: if the radii $a_i$ are generic then the
origin lies on $2n$ boundary normals (in the non-generic case
it lies on infinitely many normals). There is a conjecture that
any convex body has a caustic which is at least as complicated
as this.

\begin{Conjecture}[{\cite[Problem A3]{CFG}}] \label{conj:concurrent_normals} If $C$ is
 a convex body in $\RR^n$ $($with smooth boundary for the
 purposes of our discussion$)$, then there exists a point
 $q\in C$ such that $q$ lies on at least $2n$ inward
 normals to $\partial C$.
\end{Conjecture}

\begin{Remark}
  This conjecture is known to hold in dimensions two (where it
is easy), three (where it was proved by Heil
\cite{Heil2,Heil3,Heil1}) and dimension four (where it was
proved by Pardon~\cite{Pardon}). It is also known when $C$ is
centrally symmetric, i.e., symmetric under the antipodal map,
which was proved by Kuiper \cite{Kuiper} as a corollary of
Ljusternik--Schnirelmann theory \cite{LS}. The proofs in these
cases translate it into a problem in Morse theory via the {\em
 support function}.
\end{Remark}

\begin{Definitions} The support function $h_C$ of a convex
body $C\subset\RR^n$ is the function on the sphere
$S^{n-1}=\{v\in\RR^n\mid |v|=1\}$ defined by setting
$h_C(v) = \sup \{\langle u, v\rangle\mid u\in C \}$.
If $h_C$ is continuously differentiable (for example, if
$C$ has smooth boundary) and $\nabla h_C$ denotes its
gradient with respect to the round metric, then we can recover
$\partial C$ as the image of the parametrisation
\[\varphi\colon\ S^{n-1}\to \RR^n,\qquad \varphi(v) =
 h_C(v)v+\nabla h_C(v).\] The inward normal vector to
$\partial C$ at $\varphi(v)$ is $-v$ and the boundary
normals passing through the origin correspond to critical points
of $h_C$. If we translate $C$ so that a point $q\in C$ is
the origin, then the support function becomes
$h_{C-q}(v)=h_C(v)-\langle q,v\rangle$. This gives us a family
of functions~$h_{C-q}$ parametrised by $q\in C$, which are
Morse for generic $q\in C$. Those which are Morse are
guaranteed to have at least $2$ critical points, but the
concurrent normals conjecture asserts that at least one of them
has $2n$ critical points.
\end{Definitions}

\subsection{Reformulation in terms of surplusection}

This problem can be viewed as a Lagrangian surplusection
problem in the cotangent bundle~$T^*S^{n-1}$. Namely, the
space of oriented straight lines in $\RR^n$ is naturally a
symplectic manifold: it is the symplectic reduction (at height
$1$) of $(T^*\RR^n,\sum {\rm d}p_i\wedge {\rm d}q_i)$ by the
$\RR$-action (cogeodesic flow) generated by the Hamiltonian
$\frac{1}{2}\sum p_i^2$. Indeed, it is symplectomorphic to
$T^*S^{n-1}$: intuitively, each oriented line has a given
direction in $S^{n-1}$ and there is a tangent hyperplane's
worth of oriented lines pointing in a given direction. More
formally, the unit cotangent bundle of $\RR^n$ is strictly
contactomorphic to the 1-jet bundle of $S^{n-1}$ via Arnold's
{\em hodograph transformation}~\cite[pp.~48--49]{ArnoldULS}, and
the Reeb flow by which we symplectically reduce is the
cogeodesic flow in one case and addition of a constant to the
1-jet in the other, so that in both cases symplectic reduction
yields $T^*S^{n-1}$.

Given a submanifold $\Sigma\subset\RR^n$, the space of
lines orthogonal to it form a Lagrangian
sub\-mani\-fold\footnote{See, for example, Arnold's survey on ray
 systems \cite{ArnoldRays}, in particular, Section~3.A, Example~4. To
 see why it is Lagrangian, if we think of $T^*S^{n-1}$ as a
 symplectic reduction of $T^*\RR^n$ via geodesic flow,
 $L_\Sigma$ is the reduction of the (Lagrangian) conormal
 bundle of $\Sigma$.} $L_\Sigma$ of $T^*S^{n-1}$. For
example, if $\Sigma$ is a point $q$ then $L_\Sigma$ is
$\OP{graph}({\rm d}(\langle q,-\rangle))$. If $\Sigma=\partial C$
is the (smooth) boundary of a convex body $C$ then
$L_\Sigma=\OP{graph}({\rm d}h_C)$. Therefore, we have the following
reformulation of the concurrent normals conjecture. Consider the
family of Lagrangian submanifolds
$L_q=\OP{graph}({\rm d}\langle q,-\rangle)\subset T^*S^{n-1}$, with
$q\in C$. Since $L_q$ and $L_{\partial C}$ are Hamiltonian
deformations of the zero-section in $T^*S^{n-1}$, the
geometric intersection number $i(L_q,L_{\partial C})$ is equal
to $2$. However, we expect to find surplusection in this
family.

\begin{Conjecture}[reformulation] Let $C\subset \RR^n$
 be a convex body with smooth boundary. There is a point
 $q_*\in C$ with
 $\# (L_{q_*}\cap L_{\partial C} )=2n$.
 \end{Conjecture}

\begin{Remark} It is not possible to bound the mean
surplusection from below, because the locus of surplusection can
be arbitrarily small (e.g., for a circle there is a single $q$
which lies on {\em all} normals and every other point lies on
precisely two). Note that the conjecture still seems to be
difficult even without the restriction $q\in C$; see for
example the discussion in {\cite[Section 5, Lemma~10]{Pardon}}.\looseness=-1
\end{Remark}

The concurrent normals conjecture is therefore another
motivation to study surplusection phenomena. Given that we know
it is true for $n=2,3,4$, we feel it is not unreasonable to
ask a~``multi-valued'' analogue\footnote{Following Arnold
 \cite{ArnoldTPTWP}, we think of non-graphical Lagrangians in
 cotangent bundles as differentials of multi-valued functions.}
of the concurrent normals conjecture; note that the condition
$q\in C$ no longer makes sense in this generality.

\begin{Question}\label{pg:multivalued_cnc}
Does the surplusection associated with the concurrent normals
conjecture continue to hold for Lagrangian spheres which are not
graphical? More precisely, if $L\subset T^*S^{n-1}$ is a
Lagrangian sphere which is Hamiltonian isotopic to the
zero-section, does there exist a $q\in\RR^n$ such that
$\# (L\cap L_q)\geq 2n$?
\end{Question}

\begin{Remark} The answer is clearly yes when $n=2$: for
graphical Lagrangians $\OP{graph}({\rm d}f)$, it follows from the
concurrent normals conjecture to a convex body whose support
function is $f$; for a non-graphical Lagrangian $L$, there
is a cotangent fibre $F$ which intersects $L$
non-transversely but with total multiplicity $1$. Now by
tilting $F$ slightly and approximating it by a suitable
Lagrangian $L_q$ as in Figure~\ref{fig:multivalued_cnc}, we
obtain four intersections.
\end{Remark}

For those wishing to learn about some failed attempts to
tackle the concurrent normals conjecture (to help you avoid the
same pitfalls), see the blog of the second author.\footnote{\url{https://jde27.uk/blog/concurrent_normals.html}.}

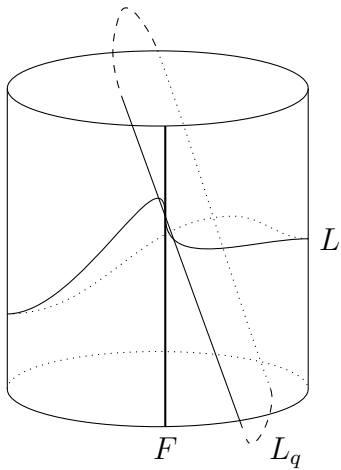
\figlet{Question \ref{pg:multivalued_cnc} has a positive answer for non-graphical Lagrangians when $n=2$. Here, $L$ is the non-graphical Lagrangian, $F$ is a cotangent fibre meeting it non-transversely, and $L_q$ is a circle from our family which approximates a tilt of $F$. The parts of $L_q$ which go very far out of the picture are drawn dashed and anything at the back of the cylinder is drawn dotted.}{fig:multivalued_cnc}{
 \draw (0,2) circle [x radius = 2, y radius=0.5];
 \draw (-2,-2) arc [x radius = 2, y radius=0.5, start angle=-180, end angle = 0];
 \draw[dotted] (-2,-2) arc [x radius = 2, y radius=0.5, start angle=180, end angle = 0];
 \draw (-2,-2) -- (-2,2);
 \draw (2,-2) -- (2,2);
 \draw (-2,-1) to[out=0,in=90] (0.1,0.3) to[out=-90,in=180] (2,0);
 \draw[dotted] (2,0) to[out=180,in=0] (1,0.3) to[out=180,in=0] (-2,-1);
 \draw[thick] (0.1,-2.499) -- (0.1,{2-0.499});
 \node at (0.1,-2.499) [below] {$F$};
 \node at (2,0) [right] {$L$};
 \draw (0.1,0.3) --++ (110:1.6) node (a) {};
 \draw (0.1,0.3) --++ (-70:3) node (b) {};
 \draw[dashed] (a.center) to[out=110,in=250] (-0.5,3) to[out=70,in=110] (0,2.5);
 \draw[dashed] (1.5,-2) to[out=290,in=-70] (b);
 \draw[dotted] (0,2.5) to[out=290,in=110] (1.5,-2);
 \node at (1.7,-2.5) [below] {$L_q$};
}

\section[Reformulation in T\^*G]{Reformulation in $\boldsymbol{T^*G}$}

In this final section, we will show that a large class of
surplusection problems can be recast in terms of a different
surplusection problem in a cotangent bundle.

\begin{Definitions}
  Suppose that $X$ admits a Hamiltonian action of a group
$G$ with equivariant moment map
$\mu\colon X\to \mathfrak{g}^*$. Then there is a Lagrangian
{\em moment correspondence} {\cite[p.\ 21]{Weinstein}}:
\[C_\mu\subset (T^*G)^-\times X^-\times X,\qquad
 \{(g,\mu(gx),x,gx)\mid g\in G, x\in X\}.\] Here, we identify
$T^*G$ with $G\times\mathfrak{g}^*$ by choosing a basis of
left-invariant 1-forms and write a superscript minus sign to
indicate that the sign of the symplectic form has been
reversed. Composition with this correspondence gives a way of
turning Lagrangian submanifolds in $X^-\times X$ into
Lagrangian submanifolds in $T^*G$. For example, a product
Lagrangian $K\times L$ turns into
\[C_\mu\circ(K\times L)=\bigl\{(g,\mu(gx))\mid g\in G,\, x\in K\cap
 g^{-1}L\bigr\}.
 \] Suppose that we are interested in studying
surplusection between $K$ and the family
$\{gL\mid g\in G\}$. Suppose moreover that $\mu$ is an
embedding, for example, if $X$ is a coadjoint orbit of
$G$. For example, if $X=\cp^n$, $G={\rm PU}(n+1)$, $L=\rp^n$,
then this is precisely the situation we studied before.\looseness=-1
\end{Definitions}

\begin{Lemma}\label{lma:reformulation} The composite
 Lagrangian $M\coloneqq C_\mu\circ (K\times L)\subset T^*G$
 is embedded and has the following property. The intersection
 $M\cap T^*_gG$ is in bijection with $K\cap g^{-1}L$.
\end{Lemma}
\begin{proof}
 Note that $C_\mu$ is the graph of the fibred coisotropic
 submanifold $G\times X\subset G\times\mathfrak{g}^*=T^*G$. A
 fibred coisotropic admits a symplectic reduction whose fibres
 are the leaves of the characteristic foliation; in this case,
 the symplectic reduction is $X^-\times X$, so $M$ is
 simply the preimage of $K\times L$ under the projection
 $G\times X\to X\times X$, $(g,x)\mapsto (gx,x)$. This
 shows that $M$ is embedded.

 To see the intersection, note that since $\mu$ is an
 embedding, $M\cap T^*_gG$ is in bijection with
 \[I=\{(g,\mu(gx),x,gx)\mid x\in K,\, gx\in L\}\subset
  T_g^*G\times X \times X.\] This set $I$ is, in turn, in
 bijection with
 $\{x\in X\mid x\in K,\, gx\in L\}=K\cap g^{-1}L$.
\end{proof}

\begin{Remark}
  Let $\pi\colon T^*G\to G$ be the projection. The
surplusection locus for the original problem is then the locus
in $G$ where $\pi|_M$ is more than $i(K,L)$-to-one,
i.e., where $M$ is ``more folded than it needs to be'' with
respect to the fibration by cotangent fibres. If we think about
what could be responsible for this folding, note that $M$ is
vertically ``cramped'': it is forced to live in the compact set
$G\times X\subset T^*G$.
\end{Remark}

\begin{Remark}
  One can use this correspondence even if $\mu$ fails to be
an embedding, but then the conclusions of Lemma~\ref{lma:reformulation} fail. However, it could still be useful:
for example, if $i(K,L)=0$, then the projection of $M$ to the
zero-section $G$ still coincides with the surplusection locus.
\end{Remark}

\subsection*{Acknowledgements} This note grew out of many
conversations with many people, including Chris Evans, Ivan
Smith, John Pardon, Umut Varolgunes, Leonid Polterovich, Kai
Hugtenburg, Jo\'{e} Brendel, Felix Schlenk, Elliot Gathercole,
Nikolas Adaloglou and Matt Buck. We are particularly grateful to
Egor Shelukhin for pointing out the work of Viterbo
\cite{ViterboMetric}, which led us to the short paper
\cite{Goldstein03} of Goldstein, both of which give even more
evidence for the ubiquity of surplusection (whilst also
rendering our original volume bound obsolete). We also thank the
referees for their prompt and insightful comments. J.E. is
supported by EPSRC Grant EP/W015749/1. G.D.R. is supported by
the Knut and Alice Wallenberg Foundation under the grants KAW
2021.0191 and KAW 2021.0300, and by the Swedish Research Council
under the grant number 2020-04426.

\pdfbookmark[1]{References}{ref}
\LastPageEnding

\end{document}